\documentclass[]{interact}

\usepackage{epstopdf}
\usepackage[caption=false]{subfig}
\usepackage{float}
\usepackage{amsmath,amsfonts,amssymb,amsthm}
\usepackage{enumerate}
\usepackage{stmaryrd}
\usepackage[utf8]{inputenc}
\usepackage[T1]{fontenc}
\usepackage[english]{babel}
\usepackage{graphicx}
\usepackage{color}
\usepackage{pgf,tikz,pgfplots}
\usepackage{lmodern}
\usepackage[colorlinks=true]{hyperref}
\hypersetup{urlcolor=blue, citecolor=blue}

\usepackage{array}
\usepackage{array,multirow,makecell}
\setcellgapes{1pt}
\makegapedcells
\newcolumntype{R}[1]{>{\raggedleft\arraybackslash }b{#1}}
\newcolumntype{L}[1]{>{\raggedright\arraybackslash }b{#1}}
\newcolumntype{C}[1]{>{\centering\arraybackslash }b{#1}}

\usepackage[numbers,sort&compress]{natbib}
\bibpunct[, ]{[}{]}{,}{n}{,}{,}
\makeatletter
\def\NAT@def@citea{\def\@citea{\NAT@separator}}
\makeatother

\theoremstyle{plain}
\newtheorem{theorem}{Theorem}[section]
\newtheorem{lemma}[theorem]{Lemma}
\newtheorem{corollary}[theorem]{Corollary}
\newtheorem{proposition}[theorem]{Proposition}

\theoremstyle{definition}

\theoremstyle{remark}

\newcounter{cst}

\def\l{\lambda}

\def\d{{\rm d}}
\def\1{{\bf 1}}
\def\be{\begin{equation}}
\def\ee{\end{equation}}

\begin{document}

\articletype{Research Article}

\title{Global asymptotic expansion for anisotropic singularly perturbed elliptic problems}

\author{
\name{David Maltese\textsuperscript{a} and Chokri Ogabi\textsuperscript{b}\thanks{\textsuperscript{b} Corresponding author. Email: Chokri.Ogabi@univ-eiffel.fr} }
\affil{\textsuperscript{a,b}LAMA, Univ. Gustave Eiffel, Univ. Paris Est Cr\'eteil, CNRS, F-77454 Marne-la-Vall\'ee, France.}
}

\maketitle

\begin{abstract}
In this paper, we study some anisotropic singular perturbations for a class of linear elliptic problems. We show a global asymptotic expansion of the solution in certain functional space.
\end{abstract}

\section*{Introduction}

In this paper, we deal with anisotropic singular perturbations, which model diffusion phenomena in anisotropic medium. A prototype of such problems is given by the following bi-dimensional equation
\begin{equation} \label{prototype2d}
-\epsilon^{2} \partial_{x_1}^{2} u_{\epsilon}-\partial_{x_2}^{2} u_{\epsilon} =f \text{, with } u_{\epsilon}=0 \text{ on } \partial \Omega.
\end{equation}
The question of interest is whether the solution $u_{\epsilon}$  of equation (\ref{prototype2d}) can be considered, in an appropriate sense, as an approximation of a solution to the degenerate equation:  
  \begin{equation*} 
-\partial_{x_2}^{2} u_{0} =f \text{, with boundary condition.}
\end{equation*}
This perturbation is clearly singular, and addressing this question involves constructing an asymptotic approximation of the solution. The goal here is to obtain an approximation over the entire domain  $\Omega$ .\\
In the singular limit \( \epsilon \to 0 \), there is a loss of regularity, accompanied by the emergence of a boundary layer phenomenon along the lateral boundary of the domain. Obtaining such a global approximation is not always straightforward. For isotropic singular perturbations, some later works have shed light on this subject. Key references include the foundational works of Vi\'sik and Lyusternik \cite{visik}, as well as those of Knowles and Messick \cite{messick}, Eckhaus and DeJaeger \cite{eckaus}, and Comstock \cite{comstock}. 
The method used involves seeking the solution as the sum of two series. The first is a sum of interior solutions weighted by successive powers of \( \epsilon \), obtained by formally inserting a series \( \sum_{i} \epsilon^i u_i \) into equation (\ref{prototype2d}), and the second is a sum of terms associated with the boundary layer, constructed by applying a scaling transformation near the boundary of the domain.  
An equivalent approach using corrector methods is described in \cite{Lions}. \\
In general, these methods are relatively easy to implement for problems where the diffusion term is represented by a simple Laplacian and the domain has a simple geometry.  
In the anisotropic case, a global zero-order development was proposed in \cite{guesmia} for an equation of type (\ref{prototype2d}) by using corrector method. An interior asymptotic development for a more general case of equation (\ref{prototype2d}) was also studied in \cite{Azouz}. In \cite{ogabimaltese1} we have proved a zero order global approximation of order $\epsilon $ when $f$ is zero on the lateral boundary of, in this case the boundary layer part is vanish. Another global approximation of order $\epsilon^{\alpha}$, $0 < \alpha < 1$ for more general data was proved in \cite{ogabimaltese2}. For other related works, we refer the reader to \cite{ogabi2}, \cite{ogabi3}, \cite{ogabi1}.

In this article we will show some global asymptotic expansion in the case where boundary layers are vanished, the global asymptotic expansion using the methods mentioned above still open for equations of type (\ref{prototype2d}). 
To motivate our study, let us begin by an elementary example,  take $f(x_1,x_2)=\sin x_1 \sin x_2$ and $\Omega = (0,1)^2$ in (\ref{prototype2d}), then the exact solution is given by 
\begin{equation*}
u_{\epsilon}(x_1,x_2)=\frac{\sin x_1 \sin x_2}{1+\epsilon^2}=\sum_{k=0}^{\infty} \epsilon^{2k}(-1)^k \sin x_1 \sin x_2 ,
\end{equation*}
thus the following asymptotic expansion
\begin{equation*}
u_{\epsilon}=u_0+\epsilon^2 u_2+...+\epsilon^{2d} u_d +O(\epsilon^{2d+2}), \text{ } u_k(x_1,x_2)=(-1)^k \sin x_1 \sin x_2, \text{ } k=0,...,d,
\end{equation*}
holds in $H^1(\Omega)$. We will show in this paper that $u_{\epsilon}$ admits an asymptotic expansion  in some suitable functional space on the global domain $\Omega$. In \cite{Azouz} the authors studied the asymptotic expansion in some local functional space defined on $\Omega' \subset \Omega$, where $\Omega'$ is far from certain parts of $\partial \Omega $.
We will deal with the following general set-up of (\ref{prototype2d}), see for instance \cite{Chipot}.
\begin{equation}
-\text{div}(A_{\epsilon }\nabla u_{\epsilon })=f~\text{%
in}~\Omega ,  \label{u-perturbed}
\end{equation}%
supplemented with the boundary condition 
\begin{equation}
u_{\epsilon }=0~\text{in}~\partial \Omega .  \label{u-perturbed-bord}
\end{equation}%
Here, $\Omega =\omega _{1}\times \omega _{2}$ where $\omega _{1}$ and $%
\omega _{2}$ are two bounded open sets of $%
\mathbb{R}
^{q}$ and $%
\mathbb{R}
^{N-q},$ with $N>q\geq 1$. We denote by $%
x=(x_{1},...,x_{N})=(X_{1},X_{2})\in 
\mathbb{R}
^{q}\times 
\mathbb{R}
^{N-q}$ i.e. we split the coordinates into two parts. With this notation we
set%
\begin{equation*}
\nabla =(\partial _{x_{1}},...,\partial _{x_{N}})^{T}=\binom{\nabla _{X_{1}}%
}{\nabla _{X_{2}}},\text{ }
\end{equation*}%
where 
\begin{equation*}
\nabla _{X_{1}}=(\partial _{x_{1}},...,\partial _{x_{q}})^{T}\text{ and }%
\nabla _{X_{2}}=(\partial _{x_{q+1}},...,\partial _{x_{N}})^{T}.
\end{equation*}%
The matrix-valued function $A=(a_{ij})_{1\leq i,j\leq N}:\Omega \rightarrow \mathcal{M}%
_{N}(\mathbb{R})$ satisfies the classical ellipticity assumptions
\begin{itemize}
\item There exists $\lambda > 0$ such that for a.e. $x \in \Omega$ 
\begin{equation}
A(x)\xi \cdot \xi \geq \lambda \left\vert \xi \right\vert ^{2}~\text{for any}~
\xi \in 
\mathbb{R}
^{N}.  \label{hypA1}
\end{equation}

\item The elements of $A$ are bounded that is 
\begin{equation}  \label{hypA2}
a_{ij}\in L^{\infty }(\Omega )~\text{for any}~ (i,j) \in \{1,2,....,N\}^2.
\end{equation}
\end{itemize}
We have decomposed $A$ into four blocks
\begin{equation*}
A=%
\begin{pmatrix}
A_{11} & A_{12} \\ 
A_{21} & A_{22}%
\end{pmatrix},
\end{equation*}
where $A_{11}$, $A_{22}$ are $q\times q$ and $(N-q)\times (N-q)$ matrices
respectively. Notice that (\ref{hypA1}) implies that $A_{22}$ satisfies the ellipticity assumption
\begin{equation}
\text{ For a.e. } x \in \Omega: A_{22}(x)\xi_2 \cdot \xi_2 \geq \lambda \left\vert \xi_2 \right\vert ^{2}~\text{for any}~
\xi_2 \in 
\mathbb{R}
^{N-q}.  \label{hypA_221}
\end{equation}
For $\epsilon \in (0,1]$, we have set 
\begin{equation*}
A_{\epsilon }= 
\begin{pmatrix}
\epsilon ^{2}A_{11} & \epsilon A_{12} \\ 
\epsilon A_{21} & A_{22}%
\end{pmatrix}.%
\end{equation*}%
For the estimation of the global rate of convergence for the continuous problem, we suppose the following additional assumptions 
\begin{equation}
\text{ The block }A_{22}\text{ depends only on }X_{2},  \label{A22}
\end{equation}
and 
\begin{equation}
\partial _{i}a_{ij}\in L^{\infty }(\Omega ),\partial _{j}a_{ij}\in
L^{\infty }(\Omega )\text{ for }i=1,...,q\text{ and }j=q+1,...,N.
\label{hypAd1}
\end{equation}
The weak formulation of the problem ($\ref{u-perturbed}$)-($\ref{u-perturbed-bord}$%
) is 
\begin{equation}
\left\{ 
\begin{array}{l}
\int_{\Omega }A_{\epsilon
}\nabla u_{\epsilon }\cdot \nabla \varphi \d x=\int_{\Omega }f\text{ }\varphi
\d x, \quad \forall \varphi \in H_{0}^{1}(\Omega )\text{\ \ \ \  }
\\ 
u_{\epsilon }\in H_{0}^{1}(\Omega ),\text{\ \ \ \ \ \ \ \ \ \ \ \ \ \ \ \ \ \
\ \ \ \ \ \ \ \ \ \ \ \ \ \ \ \ \ \ \ \ \ \ \ \ \ \ \ \ \ \ \ \ \ \ \ \ \ \
\ \ \ \ \ \ \ \ \ \ \ \ \ \ \ \ \ \ \ \ \ \ \ \ \ \ \ }%
\end{array}%
\right.   \label{u-perturbed-weak}
\end{equation}%
where the existence and uniqueness is a consequence of the assumptions $(\ref%
{hypA1})$, $(\ref{hypA2})$.
The limit problem is given by
\begin{equation}
-\text{div}_{X_{2}}(A_{22}\nabla u)=f~\text{in}~\Omega ,
\label{u-limit}
\end{equation}
supplemented with the boundary condition 
\begin{equation}
u(X_{1},\cdot )=0~\text{in}~\partial \omega _{2},~\text{for}~X_{1}\in \omega
_{1}.  \label{u-limit-bord}
\end{equation}
We recall the Hilbert spaces \cite{ogabimaltese1}
$$H_{0}^{1}(\Omega ;\omega _{2})=\left\{\begin{array}{cc} v\in L^{2}(\Omega )~\text{such that}%
~\nabla _{X_{2}}v\in L^{2}(\Omega )^{N-q} \\
\text{ and for a.e. } ~X_{1}\in
\omega _{1},~v(X_{1},\cdot )\in H_{0}^{1}(\omega _{2})
\end{array}\right\}, $$ 
$$H_{0}^{1}(\Omega ;\omega _{1})=\left\{\begin{array}{cc} v\in L^{2}(\Omega )~\text{such that}%
~\nabla _{X_{1}}v\in L^{2}(\Omega )^{q} \\
\text{ and for a.e. } ~X_{2}\in
\omega _{2},~v(\cdot,X_{2} )\in H_{0}^{1}(\omega _{1})
\end{array}\right\}, $$ 
equipped with the norms $\left\Vert \nabla _{X_{2}}(\cdot )\right\Vert
_{L^{2}(\Omega )^{N-q}}$, $\left\Vert \nabla _{X_{1}}(\cdot )\right\Vert
_{L^{2}(\Omega )^{q}}$ respectively. Notice that these norms are equivalent to 
\begin{equation*}
\left( \left\Vert (\cdot )\right\Vert _{L^{2}(\Omega )}^{2}+\left\Vert
\nabla _{X_{2}}(\cdot )\right\Vert _{L^{2}(\Omega )^{N-q}}^{2}\right) ^{1/2},\text{ and } \left( \left\Vert (\cdot )\right\Vert _{L^{2}(\Omega )}^{2}+\left\Vert
\nabla _{X_{1}}(\cdot )\right\Vert _{L^{2}(\Omega )^{q}}^{2}\right) ^{1/2},
\end{equation*}%
thanks to Poincar\'{e}'s inequality 
\begin{equation}
\left\Vert v\right\Vert _{L^{2}(\Omega )}\leq C_{\omega _{i}}\left\Vert
\nabla _{X_{i}}v\right\Vert _{L^{2}(\Omega )},~\text{for any}~v\in
H_{0}^{1}(\Omega ;\omega _{i}), \text{ } i=1,2. \label{sob}
\end{equation}%
The space $H_{0}^{1}(\Omega )$ will be normed by $\left\Vert \nabla (\cdot
)\right\Vert _{L^{2}(\Omega )^{N}}$. One can check immediately that the
embedding $H_{0}^{1}(\Omega )\hookrightarrow H_{0}^{1}(\Omega ,\omega _{2})$
is continuous. The weak formulation of the limit problem (\ref{u-limit})-(\ref{u-limit-bord}) is given by 
\begin{equation}
\left\{ 
\begin{array}{ll}
\begin{array}{c}
\int_{\omega
_{2}}A_{22}\mathbf{(}X_{1},\cdot )\nabla _{X_{2}}u\mathbf{(}X_{1},\cdot
)\cdot \nabla _{X_{2}}\psi dX_{2} 
=\int_{\omega _{2}}f\mathbf{(}X_{1},\cdot )\text{ }\psi dX_{2}\text{, }%
\forall \psi \in H_{0}^{1}(\omega _{2})\text{\ \ \ \ }%
\end{array}
&  \\ 
u\mathbf{(}X_{1},\cdot )\in H_{0}^{1}(\omega _{2})\text{,\ for a.e. }%
X_{1}\in \omega _{1}\text{, \ \ \ \ \ \ \ \ \ \ \ \ \ \ \ \ \ \ \ \ \ \ \ \ \
\ \ \ \ \ \ \ \ \ \ \ \ \ \ \ \ \ \ \ \ \ \ \ \ \ \ \ \ \ \ \ \ \ \ \ \ } & 
\text{ }%
\end{array}%
\right.   
\label{u-limit-weak-bis}
\end{equation}%
where the existence and uniqueness is a consequence of the assumptions (\ref{hypA2}),(\ref{hypA_221}).
Recall that we have $u\in H_{0}^{1}(\Omega ;\omega _{2})$ and $ u_{\epsilon }\rightarrow u\text{ in }H_{0}^{1}(\Omega ;\omega _{2})$ as $\epsilon \rightarrow 0$ \cite{ogabimaltese1}.
Notice that for $\varphi \in H_{0}^{1}(\Omega ;\omega _{2})$, then for a.e $%
X_{1}$ in $\omega _{1}$ we have $\varphi (X_{1},\cdot )\in H_{0}^{1}(\omega
_{2})$, testing with it in (\ref{u-limit-weak-bis}) and integrating over $%
\omega _{1}$ yields 
\begin{equation}
\int_{\Omega }A_{22}\nabla _{X_{2}}u\cdot
\nabla _{X_{2}}\varphi dx=\int_{\Omega }f\text{ }\varphi dx,\text{\ }\forall
\varphi \in H_{0}^{1}(\Omega ;\omega _{2}).  \label{u-limit-weak}
\end{equation}
Notice that (\ref{u-limit-weak}) could be seen as the weak formulation of the limit problem (\ref{u-limit})-(\ref{u-limit-bord}) in the Hilbert space $H_{0}^{1}(\Omega ;\omega _{2})$ (thanks to Poincaré's inequality (\ref{sob})).\\
For the estimation of the global rate of convergence, let us recall this our previous result.
\begin{theorem} \label{thm-tauxconvergence}
\cite{ogabimaltese1} Let $\Omega =\omega _{1}\times \omega _{2}$ where $\omega
_{1}$ and $\omega _{2}$ are two bounded open sets of $%
\mathbb{R}
^{q}$ and $%
\mathbb{R}
^{N-q}$ respectively, with $N>q\geq 1.$ Suppose that $%
A $ satisfies $(\ref{hypA1})$, $(\ref{hypA2})$, $(\ref{A22})$ and $(\ref%
{hypAd1})$. Let $f\in H^1_0(\Omega,\omega_1 )$, then there exists $C_{\l,\Omega,A}>0$ such that: 
\begin{equation}
\label{tauxconvergence}
\forall \epsilon \in (0,1]:\left\Vert \nabla _{X_{2}}(u_{\epsilon
}-u)\right\Vert _{L^{2}(\Omega )^{N-q}}\leq C_{\l,\Omega,A}(\left\Vert \nabla _{X_{1}}f\right\Vert
_{L^{2}(\Omega )^{q}}+\left\Vert f\right\Vert _{L^{2}(\Omega )})\ \epsilon ,
\end{equation}%
where $u_{\epsilon }$ is the unique solution of $(\ref{u-perturbed-weak})$ in $%
H_{0}^{1}(\Omega )$ and $u$ is the unique solution to $(\ref{u-limit-weak})$
in $H_{0}^{1}(\Omega ;\omega _{2}),$  moreover we have $u \in H^1_0(\Omega).$
\end{theorem}
The estimate (\ref{tauxconvergence}) could be seen as estimate for an asymptotic expansion of $u_{\epsilon}$ of order 0.
The aim of our study is to obtain analogous estimates for asymptotic expansion of order $\geq 1$. To obtain such a result we have suppose the following reularity asumption on $f$
We look for the expansion of $u_{\epsilon }$ into a formal series 
\begin{equation}
u_{\epsilon }=\sum_{k \geq 0}\epsilon ^{k}u_{k}.  \label{serie}
\end{equation}%
We replace formally (\ref{serie}) in (\ref{u-perturbed}) and we
identify the terms of the same powers, we obtain%
\begin{equation}
-\nabla _{X_{2}}\cdot (A_{22}\nabla _{X_{2}}u_{0})=f,  \label{u0}
\end{equation}%
and 
\begin{equation}
-\nabla _{X_{2}}\cdot (A_{22}\nabla _{X_{2}}u_{1})=\nabla _{X_{1}}\cdot
(A_{12}\nabla _{X_{2}}u_{0})+\nabla _{X_{2}}\cdot (A_{21}\nabla
_{X_{1}}u_{0}),  \label{u1}
\end{equation}%
and for every $k\geq 2:$%
\begin{multline}
-\nabla _{X_{2}}\cdot (A_{22}\nabla _{X_{2}}u_{k})=\nabla _{X_{1}}\cdot
(A_{12}\nabla _{X_{2}}u_{k-1}) \\
+\nabla _{X_{2}}\cdot (A_{21}\nabla
_{X_{1}}u_{k-1})+\nabla _{X_{1}}\cdot (A_{11}\nabla _{X_{1}}u_{k-2}).
\label{uk}
\end{multline}
To obtain a such asymptotic expansion the data of the problem must be regular enough, i.e. we assume that 
\begin{equation} \label{f_Hm0}
f\in H_{0}^{d+1}(\Omega ;\omega _{1}).
\end{equation}
Notice that (\ref{u0}), (\ref{u1}) and (\ref{uk}) are understood in the weak sense in the functional $H^1_0(\Omega, \omega_2)$ as in (\ref{u-limit-weak}).\\
For every multi-index $\alpha =(\alpha _{1},...,\alpha
_{q})\in 
\mathbb{N}
^{q}$ and $\beta =(\beta _{q+1},...,\beta
_{N})\in 
\mathbb{N}
^{N-q}$, we denote $$D_{X_{1}}^{\alpha }:=\frac{%
\partial ^{\alpha _{1}....\alpha _{q}}}{\partial x_{1}^{\alpha
_{1}}...\partial x_{q}^{\alpha _{q}}}, \text{ and } D_{X_{2}}^{\beta }:=\frac{%
\partial ^{\beta _{q+1}....\beta _{N}}}{\partial x_{q+1}^{\beta
_{q+1}}...\partial x_{N}^{\beta _{N}}}, $$
and for every $1 \leq i \leq q$ and $q+1 \leq j \leq N$ we denote generically $$ D^1_{X_1}:=\frac{\partial_i}{\partial x_i} \text{ and }D^1_{X_2}:=\frac{\partial_j}{\partial x_j} .$$
For any $m\in 
\mathbb{N}
^{\ast }$, we introduce the space,%
\begin{equation*}
H_{0}^{m}(\Omega ;\omega _{1})=\left\{\begin{array}{cc} v\in L^{2}(\Omega )\mid
D_{X_{1}}^{\alpha }v\in L^{2}(\Omega )\text{ for every }\alpha \in 
\mathbb{N}
^{q}\text{ },\left\vert \alpha \right\vert
\leq m, \\ \text{ and for a.e. }X_{2}\in \omega _{2},u(\cdot ,X_{2})\in
H_{0}^{m}(\omega _{1})\end{array}\right\}.
\end{equation*}
We recall that $H_{0}^{m}(\omega _{1})$ is the closure of $\mathcal{D}%
(\omega _{1})$ in the Sobolev space $H^{m}(\omega _{1}).$ The space $H_{0}^{m}(\Omega ;\omega _{1})$ is
normed by 
\begin{equation*}
\left( \sum_{\substack{ \alpha \in 
\mathbb{N}
^{q} \\ \left\vert \alpha \right\vert \leq m 
}}\left\Vert D_{X_{1}}^{\alpha } (\cdot) \right\Vert _{L^{2}(\Omega )}^{2}\right) ^{%
\frac{1}{2}},
\end{equation*}%
or with the equivalent norm
\begin{equation*}
 \left( \sum_{\substack{ \alpha \in 
\mathbb{N}
^{q} \\ \left\vert \alpha \right\vert =m}}%
\left\Vert D_{X_{1}}^{\alpha }(\cdot) \right\Vert _{L^{2}(\Omega )}^{2}\right) ^{%
\frac{1}{2}}.
\end{equation*}
We remark that $H_{0}^{m}(\Omega ;\omega _{1}) \subset H_{0}^{1}(\Omega ;\omega _{1})$, and we can check that $H_{0}^{m}(\Omega ;\omega _{1})$ is a Hilbert space. \\
The article is organized as follows, 
 in section 1, we show the asymptotic expansion in the case when $A$ has a diagonal block structure. In section 2, we extend the result to the general case. We have also put the proofs of some density lemmas in the appendix. Throughout this manuscript, we denote $C_{f,A,...ect}$ a generic positive constant (not necessary the same each time) independent of $\epsilon$ and which depends on the objects $A$, $f$,...etc.
\section{The diagonal block structure }\label{sec1}
In this subsection, we suppose that $d \in \mathbb{N}^{*}$ to be even, and we assume that
\begin{equation}\label{block0}
 A_{12}=A_{21}=0.
\end{equation}
The equations (\ref{u0}), (\ref{u1}), and (\ref{uk}) read 
\begin{equation}
-\nabla _{X_{2}}\cdot (A_{22}\nabla _{X_{2}}u_{0})=f  \label{u0-bis},
\end{equation}%
and 
\begin{equation}
-\nabla _{X_{2}}\cdot (A_{22}\nabla _{X_{2}}u_{1})=0  \label{u1-bis},
\end{equation}%
and for every $k\geq 2:$%
\begin{equation}
-\nabla _{X_{2}}\cdot (A_{22}\nabla _{X_{2}}u_{k})=\nabla _{X_{1}}\cdot (A_{11}\nabla _{X_{1}}u_{k-2}).
\label{uk-bis}
\end{equation}
Let $d\in 
\mathbb{N}^{*}
.$ We assume that: 
\begin{equation}
D_{X_{1}}^{\alpha }A_{11}\in L^{\infty }(\Omega ), \text{ for every } \alpha \in 
\mathbb{N}
^{q} \text{ with } \left\vert \alpha \right\vert
\leq d,  \label{regularity -A11}
\end{equation}
then we have the following
\begin{theorem} \label{thm-1}
Assume that (\ref{hypA1}), (\ref{hypA2}), (\ref{A22}), (\ref{f_Hm0}), (\ref{block0}), and (\ref{regularity -A11}) are satisfied. There exit $u_0$, ...,$u_d$ in $H^1_0(\Omega)$ the unique solution to the elliptic system (\ref{u0-bis}), (\ref{u1-bis}), (\ref{uk-bis}) such that for 
every $\epsilon \in (0,1]$ 
\begin{equation*}
\left\Vert \nabla _{X_{2}}(u_{\epsilon }-\sum_{k=0}^{\frac{d}{2}}\epsilon
^{2k}u_{2k})\right\Vert _{L^{2}(\Omega )}\leq C \times \epsilon ^{d+1}, \text{ and } \left\Vert \nabla _{X_1}(u_{\epsilon }-\sum_{k=0}^{\frac{d}{2}}\epsilon
^{2k}u_{2k})\right\Vert _{L^{2}(\Omega )}\leq C \times \epsilon ^{d}.
\end{equation*}%
\end{theorem}

\subsection{\textbf{Regularity of the $u_{k}$}}
We have the following lemma
\begin{lemma}
\label{lem:regularity}Let $m\in 
\mathbb{N}
^{\ast }.$ Let $g\in H_{0}^{m}(\Omega ;\omega _{1}).$ Assume (\ref{hypA1}), (\ref{hypA2})), (\ref{A22}),  (\ref{regularity -A11}), and (\ref{f_Hm0}). Let $w_{g}\in H_{0}^{1}(\Omega ;\omega _{2})$
be the unique solution of 
\begin{equation}
\int_{\Omega }A_{22}\nabla _{X_{2}}w_{g}\cdot \nabla _{X_{2}}\varphi
dx=\int_{\Omega }g\varphi dx,\text{ }\forall \varphi \in H_{0}^{1}(\Omega
;\omega _{2}),  \label{regul1}
\end{equation}%
then $w_{g}\in H_{0}^{m}(\Omega ;\omega _{1}) \cap H^1_0(\Omega)$.
\end{lemma}
We have this immediate result
\begin{corollary}
\label{coro:regularity}
Suppose that assumptions of Theorem \ref{thm-1} are satisfied, then the system $(\ref{u0-bis}),(\ref{u1-bis}%
),(\ref{uk-bis})$ has a unique solution in $(H_{0}^{1}(\Omega ;\omega
_{2}))^{d}$. Moreover we have $u_k\in
H_{0}^{1}(\Omega ),$ and $u_k\in H^{d-k+1}_0(\Omega,\omega_1) $ for $0\leq k \leq d $
\end{corollary}
\begin{proof}
The existence and the uniqueness of $u_k$ in $H^1_0(\Omega,\omega_2)$ follow from the Lax-Milgram theorem. For the regularity, we process by induction.

1) For $u_{0}:$  by using (\ref{f_Hm0}) and Lemma \ref{lem:regularity} we get immediately $H_{0}^{m}(\Omega ;\omega _{1})$,
the regularity follows from Lemma , we have $u_{0}\in H_{0}^{d+1}(\Omega ;\omega _{1}).$

2) For $u_{k}$ with $1 \leq k\leq d-1$ odd : From (\ref{u1-bis}) we get immediately $u_1=0$, and by using (\ref{uk-bis}) and induction we get $u_k=0$ for every $0 \leq k \leq d$ odd.

3) For $u_{k}$ with $2 \leq k\leq d-2$ even: Suppose $u_{k-2} \in H_{0}^{d-k+3}(\Omega ;\omega
_{1})$ then by using (\ref{regularity -A11}) we get $\nabla _{X_{1}}\cdot (A_{11}\nabla _{X_{1}}u_{k-2}) \in H_{0}^{d-k+1}(\Omega ;\omega
_{1})$, therefore by using Lemma \ref{lem:regularity} we obtain that $u_k \in H_{0}^{d-k+1}(\Omega ;\omega
_{1})$, finally by using 1) and induction we get the desired result.
\end{proof}
We will prove Lemma \ref{lem:regularity} in two steps. Foremost, existence and uniqueness of $w_g$ follow by Lax-Milgram theorem, and the affirmation $w_g \in H^1_0(\Omega)$ follows from Theorem \ref{thm-tauxconvergence}.  \\
\textbf{Step1:} \textit{The case $g\in H_{0}^{m}(\protect\omega _{1})\otimes
H_{0}^{m}(\protect\omega _{2})$
}
\begin{lemma}
\label{tensor} Suppose that $g\in H_{0}^{m}(\omega _{1})\otimes
H_{0}^{m}(\omega _{2}).$ Suppose that assumptions of Lemma \ref{lem:regularity} is satisfied. Let $w_{g}$ be
the unique solution of $(\ref{regul1})$ then $w_{g}\in H_{0}^{m}(\Omega
;\omega _{1})$ and $\left\Vert D_{X_{1}}^{\alpha }w_{g}\right\Vert
_{L^{2}(\Omega )}\leq C_{\omega _{2}}\left\Vert D_{X_{1}}^{\alpha
}g\right\Vert _{L^{2}(\Omega )}$ for every $\alpha \in 
\mathbb{N}
^{q}$ $,\left\vert \alpha \right\vert \leq m,$
\end{lemma}

\begin{proof}
\ \ \\
1) Suppose $g=g_{1}\otimes g_{2}.$\\
Let $w_{2,g_{2}}\in H_{0}^{1}(\omega _{2})$ be the unique solution to 
\begin{equation*}
\int_{\omega _{2}}A_{22}\nabla _{X_{2}}w_{2,g_{2}}\cdot \nabla
_{X_{2}}\varphi _{2}dx_{2}=\int_{\omega _{2}}g_{2}\varphi _{2}dx_{2},\text{ }%
\forall \varphi _{2}\in H_{0}^{1}(\omega _{2}).
\end{equation*}
For every $\varphi _{1}\in H_{0}^{1}(\omega _{1})$ we have :%
\begin{eqnarray*}
\int_{\Omega }A_{22}\nabla _{X_{2}}(g_{1}\otimes w_{2,g_{2}})\cdot \nabla
_{X_{2}}(\varphi _{1}\otimes \varphi _{2})dx &=&\int_{\omega
_{1}}g_{1}\varphi _{1}\left( \int_{\omega _{2}}A_{22}\nabla
_{X_{2}}w_{2,g_{2}}\cdot \nabla _{X_{2}}\varphi _{2}dX_{2}\right) dX_{1} \\
&=&\int_{\omega _{1}}g_{1}\varphi _{1}\left( \int_{\omega _{2}}g_{2}\varphi
_{2}dX_{2}\right) dX_{1} \\
&=&\int_{\Omega }g\varphi _{1}\otimes \varphi _{2}dx.
\end{eqnarray*}
By density of $H_{0}^{1}(\omega _{1})\otimes H_{0}^{1}(\omega _{2})$ in $%
H_{0}^{1}(\Omega ;\omega _{2})$ ( see \cite{ogabimaltese1}), then for every $%
\varphi \in H_{0}^{1}(\Omega ;\omega _{2})$ we have
\begin{equation*}
\int_{\Omega }A_{22}\nabla _{X_{2}}(g_{1}\otimes w_{2,g_{2}})\cdot \nabla
_{X_{2}}\varphi dx=\int_{\Omega }g\varphi dx,
\end{equation*}%
whence we get $w_{g}=g_{1}\otimes w_{2,g_{2}}$ ( by uniqueness). Now, it is clear
that $w_{g}\in H_{0}^{m}(\Omega ;\omega _{1})$. Indeed, for$,\alpha \in 
\mathbb{N}
^{q}\times $ $,\left\vert \alpha \right\vert \leq m$%
, we have $D_{X_{1}}^{\alpha }w_{g}=D_{X_{1}}^{\alpha }g_{1}\otimes
w_{2,g_{2}}\in L^{2}(\omega _{1})\otimes L^{2}(\omega _{2})\subset
L^{2}(\Omega ),$ and for every. $X_{2}\in \omega _{2}:w_{g}(\cdot
,X_{2})=w_{2,g_{2}}(X_{2})g_{1}(\cdot )\in H_{0}^{m}(\omega _{1}).$\\
2) Since every element $g$ of $H_{0}^{m}(\omega _{1})\otimes
H_{0}^{m}(\omega _{2})$ could be written as $g=\sum\limits_{i=1}^{l}g_{1,i}%
\otimes g_{2,i}$ for some $l\in 
\mathbb{N}
^{\ast },$ with $g_{1,i}\in H_{0}^{m}(\omega _{1})$, $g_{2,i}\in
H_{0}^{m}(\omega _{2})$, $i=1,...,l,$ then we check immediately that 
\begin{equation*}
w_{g}=\sum\limits_{i=1}^{l}g_{1,i}\otimes w_{2,i},
\end{equation*}%
where $w_{2,i}$ are the solutions of 
\begin{equation*}
\int_{\omega _{2}}A_{22}\nabla _{X_{2}}w_{2,i}\cdot \nabla _{X_{2}}\varphi
_{2}dX_{2}=\int_{\omega _{2}}g_{2,i}\varphi _{2}dX_{2},\text{ }\forall
\varphi _{2}\in H_{0}^{1}(\omega _{2})\text{, }i=1,...,l.
\end{equation*}%
We have immediately $w_{g}\in H_{0}^{m}(\Omega ;\omega _{1})$, and 
\begin{equation}
\int_{\omega _{2}}A_{22}\nabla _{X_{2}}D_{X_{1}}^{\alpha }w_{g}\cdot \nabla _{X_{2}}\varphi
_{2}dx=\int_{\omega _{2}}D_{X_{1}}^{\alpha }g\varphi _{2}dx,\text{ }\forall
\varphi _{2}\in H_{0}^{1}(\Omega,\omega _{2}).
\label{faible-tensor}
\end{equation}%
It is clear that $D_{X_{1}}^{\alpha }w_{g}\in H_{0}^{1}(\Omega,\omega _{2})$, therefore testing in (\ref{faible-tensor}) we obtain
\begin{equation}\label{nablaDX1}
\left\Vert \nabla_{X_2} D_{X_{1}}^{\alpha }w_{g}\right\Vert _{L^{2}(\Omega )}  \leq C_{\omega _{2}}\left\Vert D_{X_{1}}^{\alpha }g\right\Vert _{L^{2}(\Omega
)},
\end{equation}
and by using(\ref{sob}) we obtain immediately 
\begin{equation*}
\left\Vert D_{X_{1}}^{\alpha }w_{g}\right\Vert _{L^{2}(\Omega )}  \leq C_{\omega _{2}}\left\Vert D_{X_{1}}^{\alpha }g\right\Vert _{L^{2}(\Omega
)}.
\end{equation*}
\end{proof}
\textbf{Step 2:} \textit{The case $g\in H_{0}^{m}(\Omega ;\protect\omega _{1})$}.\\
Let $g\in H_{0}^{m}(\Omega ;\omega _{1})$. By using the density of $H_{0}^{m}(\omega _{1})\otimes H_{0}^{m}(\omega _{2})$ in $H_{0}^{m}(\Omega ;\omega
_{1})$ (see Lemmas \ref{density1} and \ref{density2}), there exists a
sequence $g_{n}$ of $H_{0}^{m}(\omega _{1})\otimes H_{0}^{m}(\omega _{2})$
such that $\left\Vert D_{X_{1}}^{\alpha }(g_{n}-g)\right\Vert _{L^{2}(\Omega
)}\rightarrow 0$ as $n\rightarrow \infty ,$ for every $\alpha \in 
\mathbb{N}
^{q}$ $,\left\vert \alpha \right\vert \leq m.$ 
According to Lemma \ref{tensor} we have $w_{g_{n}}\in H_{0}^{m}(\Omega
;\omega _{1})$, and 
\begin{equation*}
\forall \alpha \in 
\mathbb{N}
^{q},\left\vert \alpha \right\vert \leq
m,\forall n\in 
\mathbb{N}
:\left\Vert D_{X_{1}}^{\alpha }w_{g_{n}}\right\Vert _{L^{2}(\Omega )}\leq
C_{\omega _{2}}\left\Vert D_{X_{1}}^{\alpha }g_{n}\right\Vert _{L^{2}(\Omega
)}.
\end{equation*}
The linearity of (\ref{regul1}) show that,$:$%
\begin{equation*}
\forall \alpha \in 
\mathbb{N}
^{q},\left\vert \alpha \right\vert \leq
m,\forall n,p\in 
\mathbb{N}
:\left\Vert D_{X_{1}}^{\alpha }(w_{g_{n}}-w_{g_{p}})\right\Vert
_{L^{2}(\Omega )}\leq C_{\omega _{2}}\left\Vert D_{X_{1}}^{\alpha
}(g_{n}-g_{p})\right\Vert _{L^{2}(\Omega )}.
\end{equation*}%
Whence, since $(g_{n})$ \ is a Cauchy sequence in $H_{0}^{m}(\Omega ;\omega
_{1})$ then it follows from the above inequality that $(w_{g_{n}})$ is a
Cauchy sequence in $H_{0}^{m}(\Omega ;\omega _{1})$ and hence, by
completness of $H_{0}^{m}(\Omega ;\omega _{1})$ the sequence $(w_{g_{n}})$
convergences in $H_{0}^{m}(\Omega ;\omega _{1}).$ Moreover,\ for $\alpha $
fixed $\left\Vert D_{X_{1}}^{\alpha }(g_{n}-g)\right\Vert _{L^{2}(\Omega
)}\rightarrow 0$ implies $\left\Vert g_{n}-g\right\Vert _{L^{2}(\Omega
)}\rightarrow 0$, therefore, $w_{g_{n}}\rightarrow w_{g}$ in $%
H_{0}^{1}(\Omega ;\omega _{2}),$ in particular $w_{g_{n}}\rightarrow w_{g}$
in $L^{2}(\Omega )$. Whence, we get $w_{g}\in H_{0}^{m}(\Omega ;\omega _{1})$
and $\left\Vert D_{X_{1}}^{\alpha }(w_{g_{n}}-w_{g})\right\Vert
_{L^{2}(\Omega )}\rightarrow 0.$

\subsection{\textbf{Proof of Theorem \ref{thm-1}}}

By using (\ref{u-perturbed-weak}) we write, for every $\varphi \in
H_{0}^{1}(\Omega )$:%
\begin{multline*}
\epsilon ^{2}\int_{\Omega }A_{11}\nabla _{X_1}(u_{\epsilon
}-\sum_{k=0}^{d/2}\epsilon ^{2k}u_{2k})\cdot \nabla _{X_1}\varphi
dx+\int_{\Omega }A_{22}\nabla_{X_2}(u_{\epsilon }-\sum_{k=0}^{d/2}\epsilon
^{2k}u_{2k})\cdot \nabla _{X_2}\varphi dx \\
=\int_{\Omega }f\varphi dx-\int_{\Omega }A_{22}\nabla _{X_2}u_{0}\cdot \nabla
_{X_2}\varphi dx-\sum_{k=1}^{d/2}\epsilon ^{2k}\int_{\Omega }(A_{11}\nabla
_{X_1}u_{2k-2}\cdot \nabla _{X_1}\varphi +A_{22}\nabla _{X_2}u_{2k}\cdot \nabla
_{X_2}\varphi )dx \\
-\epsilon ^{d+2}\int_{\Omega }A_{11}\nabla _{X_1}u_{d}\cdot
\nabla _{X_1}\varphi dx.
\end{multline*}%
Now, by using $(\ref{u0-bis}),$ $(\ref{u1-bis}),$ (\ref{uk-bis}) we obtain:
\begin{equation*}
\epsilon ^{2}\int_{\Omega }\nabla _{X_1}(u_{\epsilon
}-\sum_{k=0}^{d/2}\epsilon ^{2k}u_{2k})\cdot \nabla _{X_1}\varphi
dx+\int_{\Omega }\nabla _{X_2}(u_{\epsilon }-\sum_{k=0}^{d/2}\epsilon
^{2k}u_{2k})\cdot \nabla _{X_2}\varphi dx=-\epsilon ^{d+2}\int_{\Omega
} A_{11}\nabla _{X_1}u_{d}\cdot \nabla _{X_1}\varphi dx.
\end{equation*}%
Since $f\in H_{0}^{2d+1}(\Omega ;\omega _{1})$ then according to Corollary 
\ref{coro:regularity} the function $(u_{\epsilon }-\sum_{k=0}^{d}\epsilon
^{2k}u_{2k})$ belongs to $ H_{0}^{1}(\Omega ),$ so testing with it in the above
equality, and using assumption (\ref{hypA1}) we derive%
\begin{multline*}
\lambda \epsilon ^{2}\left\Vert \nabla _{X_1}(u_{\epsilon }-\sum_{k=0}^{d/2}\epsilon
^{2k}u_{2k})\right\Vert _{L^{2}(\Omega )}^{2}+\left\Vert \nabla
_{X_2}(u_{\epsilon }-\sum_{k=0}^{d/2}\epsilon ^{2k}u_{2k})\right\Vert
_{L^{2}(\Omega )}^{2} \\
\leq \epsilon ^{d+2}\left\vert \int_{\Omega } A_{11}\nabla
_{X_1}u_{d}\cdot \nabla _{X_1}(u_{\epsilon }-\sum_{k=0}^{d/2}\epsilon
^{2k}u_{2k})dx\right\vert .
\end{multline*} 
\begin{equation}
\label{test}
\end{equation} 
By using Cauchy-Schwartz inequality  and Young's inequality we derive%
\begin{multline*}
\epsilon ^{d+2}\left\vert \int_{\Omega }A_{11}\nabla _{X_1}u_{d}\cdot \nabla
_{X_1}(u_{\epsilon }-\sum_{k=0}^{d/2}\epsilon ^{2k}u_{2k})dx\right\vert \leq
\frac{1}{2\lambda}\epsilon ^{2d+2}\Vert A_{11} \Vert _{\infty}^2\left\Vert \nabla _{X_1}u_{d}\right\Vert _{L^{2}(\Omega)}^2  \\ +\frac{\lambda}{2} \epsilon^2\Vert\nabla _{X_1} u_{\epsilon }-\sum_{k=0}^{d/2}\epsilon ^{2k}u_{2k} \Vert_{L^2(\Omega)}^2.
\end{multline*}%
Therefore 
\begin{multline*}
\lambda \epsilon ^{2}\left\Vert \nabla _{X_1}(u_{\epsilon }-\sum_{k=0}^{d/2}\epsilon
^{2k}u_{2k})\right\Vert _{L^{2}(\Omega )}^{2}+\left\Vert \nabla
_{X_2}(u_{\epsilon }-\sum_{k=0}^{d/2}\epsilon ^{2k}u_{2k})\right\Vert
_{L^{2}(\Omega )}^{2} \\
\leq \frac{1}{2\lambda}\epsilon ^{2d+2}\Vert A_{11} \Vert _{\infty}\left\Vert \nabla _{X_1}u_{d}\right\Vert _{L^{2}(\Omega)}.
\end{multline*}
Whence, we obtain that for every $\epsilon \in (0,1]$%
\begin{equation*}
\epsilon ^{2}\left\Vert \nabla _{1}(u_{\epsilon }-\sum_{k=0}^{d/2}\epsilon
^{2k}u_{2k})\right\Vert _{L^{2}(\Omega )}^{2}+\left\Vert \partial
_{2}(u_{\epsilon }-\sum_{k=0}^{d/2}\epsilon ^{2k}u_{2k})\right\Vert
_{L^{2}(\Omega )}^{2}\leq C_{d,f,A,\lambda}\epsilon ^{2d+2}.
\end{equation*}%
Whence 
\begin{equation*}
\left\Vert \nabla _{X_2}(u_{\epsilon }-\sum_{k=0}^{d/2}\epsilon
^{2k}u_{2k})\right\Vert _{L^{2}(\Omega )}\leq  C_{d,f,A,\lambda}\epsilon ^{d+1},
\end{equation*}%
and 
\begin{equation*}
\left\Vert \nabla _{X_1}(u_{\epsilon }-\sum_{k=0}^{d/2}\epsilon
^{2k}u_{2k})\right\Vert _{L^{2}(\Omega )} C_{d,f,A,\lambda}\epsilon ^{d}.
\end{equation*}
\section{The general case}
In this section, we deal with (\ref{u0}), (\ref{u1}), and (\ref{uk}) with general $A$ with the following regularity assumption
\begin{equation}
D_{X_{1}}^{\alpha }A_{12}\in L^{\infty }(\Omega ), \text{ }D_{X_{1}}^{\alpha }A_{21}\in L^{\infty }(\Omega ) \text{, for every } \alpha \in 
\mathbb{N}
^{q} \text{ with } \left\vert \alpha \right\vert,
\leq d-1.  \label{regularity -A12}
\end{equation}
 and we suppose that $d \in \mathbb{N}^{*}$ not necessarily even. 
\begin{lemma}\label{lem:regularity2}
Suppose that assumptions of Lemma \ref{lem:regularity} hold then we have 
 $$ D^1_{X_1}D^1_{X_2} w_{g} \in H^{m-1}_0(\Omega,\omega_1).$$
\end{lemma}
\begin{proof}
Let $\beta \in \mathbb{N}^q$ be a multi-index such that $\vert \beta \vert \leq m-1$ there exits $\alpha \in \mathbb{N}^q$ such that $$\vert \alpha \vert =\vert \beta \vert +1 \text{ and } D^{\alpha}_{X_1}=D^{\beta}_{X_1}D^{1}_{X_1}.$$
Let $g \in H^m_0(\omega_1) \otimes  H^m_0(\omega_2)$, by using (\ref{nablaDX1}) we get 
\begin{equation*}
\left\Vert  D_{X_{1}}^{\beta } D^1_{X_1} D^1_{X_2}w_{g}\right\Vert _{L^{2}(\Omega )}  \leq C_{\omega _{2}}\left\Vert D_{X_{1}}^{\alpha }g\right\Vert _{L^{2}(\Omega
)}.
\end{equation*}
It is clear that $D^1_{X_1} D^1_{X_2}w_{g} \in H^{m-1}_0(\Omega, \omega_1)$, by using a density argument as in step 2 of the proof of Lemma \ref{lem:regularity} we obtain the desired result.
\end{proof}
We have the following regularity result
\begin{proposition}\label{prop-regul}
Suppose that assumptions of Lemma \ref{lem:regularity} are satisfied, then the system (\ref{u0}), (\ref{u1}), and (\ref{uk}), $2 \leq k\leq d$, have a unique solution in $H^1_0(\Omega, \omega_2)^{d+1}$. Moreover, we have $u_k \in H^1_0(\Omega) \cap H^{d-k+1}_0(\Omega, \omega_1)$, for every $0 \leq k \leq d $.
\end{proposition}
\begin{proof}
\ \\
1) The affirmation $u_0 \in H^1_0(\Omega) \cap H^{d+1}_0(\Omega,\omega_1)$ follows as in Corollary \ref{coro:regularity}.\\
2) For $u_1$: In fact Lemma \ref{lem:regularity2} shows that 
$$ D^1_{X_1}D^1_{X_2} w_{g} \in H^{d}_0(\Omega,\omega_1).$$
Therefore, by using assumption (\ref{regularity -A12}) we show that the right hand side of (\ref{u1}) belongs to $H^{d}_0(\Omega,\omega_1)$, whence by using Lemmas \ref{lem:regularity},\ref{lem:regularity2} we get immediately $$u_1 \in H^1_0(\Omega) \cap H^{d}_0(\Omega, \omega_1), \text{ and } D^1_{X_1}D^1_{X_2} u_{1} \in H^{d-1}_0(\Omega,\omega_1).$$
3) For $u_k$: Let $2 \leq k \leq d$, suppose that $u_{k-1} \in H^{d-k+2}_0(\Omega, \omega_1)$, $u_{k-2} \in H^{d-k+3}_0(\Omega, \omega_1)$ and $ D^1_{X_1}D^1_{X_2} u_{k-1} \in H^{d-k+1}_0(\Omega,\omega_1)$, then assumptions (\ref{regularity -A11}), (\ref{regularity -A12})show that the right hand side of (\ref{uk}) belongs to $H^{d-k+1}_0(\Omega,\omega_1)$, whence by using Lemmas \ref{lem:regularity},\ref{lem:regularity2} we get immediately
$$u_k \in H^1_0(\Omega) \cap H^{d-k+1}_0(\Omega, \omega_1), \text{ and } D^1_{X_1}D^1_{X_2} u_{k} \in H^{d-k}_0(\Omega,\omega_1).$$
Finally, by induction we get the suspected result.
\end{proof}
Now, we are ready to give the main theorem of this Section .
\begin{theorem}\label{thm-2}
Assume that (\ref{hypA1}), (\ref{hypA2})), (\ref{A22}),  (\ref{regularity -A11}), (\ref{f_Hm0}),and (\ref{regularity -A12}) are satisfied. There exit $u_0$, ...,$u_d$ in $H^1_0(\Omega)$ solution to the elliptic system (\ref{u0}), \ref{u1}), (\ref{uk})such that 
every $\epsilon \in (0,1]$ 
\begin{equation*}
\left\Vert \nabla _{X_{2}}(u_{\epsilon }-\sum_{k=0}^{d}\epsilon
^{k}u_{k})\right\Vert _{L^{2}(\Omega )}\leq C \times \epsilon ^{d}, \text{ and } \left\Vert \nabla _{X_1}(u_{\epsilon }-\sum_{k=0}^{d}\epsilon
^{k}u_{k})\right\Vert _{L^{2}(\Omega )}\leq C \times \epsilon ^{d-1}.
\end{equation*}%
\end{theorem}
\begin{proof}
By using (\ref{u-perturbed-weak}) we compute, for every $\varphi \in
H_{0}^{1}(\Omega )$:
\begin{multline*}
\int_{\Omega } A_{\epsilon}\nabla \Big(u_{\epsilon
}-\sum_{k=0}^{d}\epsilon ^{k}u_{k}\Big)\cdot \nabla \varphi dx =\\
\epsilon ^{2}\int_{\Omega }A_{11}\nabla _{X_1}(u_{\epsilon
}-\sum_{k=0}^{d}\epsilon ^{k}u_{k})\cdot \nabla _{X_1}\varphi
dx+\int_{\Omega }A_{22}\nabla_{X_2}(u_{\epsilon }-\sum_{k=0}^{d}\epsilon
^{k}u_{k})\cdot \nabla _{X_2}\varphi dx \\
+\epsilon\int_{\Omega }A_{12}\nabla _{X_2}(u_{\epsilon
}-\sum_{k=0}^{d}\epsilon ^{k}u_{k})\cdot \nabla _{X_1}\varphi
dx+\epsilon\int_{\Omega }A_{21}\nabla_{X_1}(u_{\epsilon }-\sum_{k=0}^{d}\epsilon
^{k}u_{k})\cdot \nabla _{X_2}\varphi dx \\
=\int_{\Omega }f\varphi dx-\int_{\Omega }A_{22}\nabla _{X_2}u_{0}\cdot \nabla
_{X_2}\varphi dx-\epsilon\Big(\int_{\Omega }A_{22}\nabla _{X_2}u_{1} \cdot \nabla _{X_2} \varphi dx+\int_{\Omega } A_{12}\nabla _{X_2}u_{0} \cdot \nabla _{X_1} \varphi dx \\ +\int_{\Omega } A_{21}\nabla _{X_1}u_{0} \cdot \nabla _{X_2} \varphi dx \Big) 
-\sum_{k=2}^{d}\epsilon ^{k}\int_{\Omega }\Big(A_{11}\nabla
_{X_1}u_{k-2}\cdot \nabla _{X_1}\varphi +A_{22}\nabla _{X_2}u_{k}\cdot \nabla
_{X_2}\varphi + \\ A_{12}\nabla _{X_2}u_{k-1} \cdot \nabla _{X_1} \varphi + A_{21}\nabla _{X_1}u_{k-1} \cdot \nabla _{X_2} \varphi  \Big)dx
-\epsilon ^{d+2}\int_{\Omega }A_{11}\nabla _{X_1}u_{d}\cdot
\nabla _{X_1}\varphi dx \\
-\epsilon^{d+1}\int_{\Omega }A_{11}\nabla _{X_1}u_{d-1}\cdot \nabla _{X_1}\varphi dx -\epsilon^{d+1}\int_{\Omega }A_{12}\nabla_{X_2} u_d \cdot \nabla_{X_1} \varphi dx - \epsilon^{d+1}\int_{\Omega }A_{21}\nabla_{X_1} u_d \cdot \nabla_{X_2} \varphi dx.
\end{multline*}%
\begin{equation}\label{eq-long}
\end{equation}
We combine (\ref{u0}), (\ref{u1}), (\ref{uk}) with (\ref{eq-long})we get 
\begin{multline*}
\int_{\Omega } A_{\epsilon}\nabla \Big(u_{\epsilon
}-\sum_{k=0}^{d}\epsilon ^{k}u_{k}\Big)\cdot \nabla \varphi dx =
-\epsilon ^{d+2}\int_{\Omega }A_{11}\nabla _{X_1}u_{d}\cdot
\nabla _{X_1}\varphi dx \\
-\epsilon^{d+1}\int_{\Omega }A_{11}\nabla _{X_1}u_{d-1}\cdot \nabla _{X_1}\varphi dx -\epsilon^{d+1}\int_{\Omega }A_{12}\nabla_{X_2} u_d \cdot \nabla_{X_1} \varphi dx - \epsilon^{d+1}\int_{\Omega }A_{21}\nabla_{X_1} u_d \cdot \nabla_{X_2} \varphi dx.
\end{multline*}
\begin{equation}\label{eq-long-bis}
\end{equation}
The function $u_{\epsilon
}-\sum_{k=0}^{d}\epsilon ^{k}u_{k}$ belongs to $H^1_0(\Omega)$, thanks to Proposition \ref{prop-regul}, so testing with the above function in (\ref{eq-long-bis}), and using the ellipticity assumption (\ref{hypA1}) we get 
\begin{multline*}
\lambda \epsilon ^{2}\left\Vert \nabla _{X_1}(u_{\epsilon }-\sum_{k=0}^{d}\epsilon
^{k}u_{k})\right\Vert _{L^{2}(\Omega )}^{2}+ \lambda\left\Vert \nabla
_{X_2}(u_{\epsilon }-\sum_{k=0}^{d}\epsilon ^{k}u_{k})\right\Vert
_{L^{2}(\Omega )}^{2} 
\leq \\ \underbrace{-\epsilon ^{d+2}\int_{\Omega }A_{11}\nabla _{X_1}u_{d}\cdot
\nabla _{X_1}\Big(u_{\epsilon
}-\sum_{k=0}^{d}\epsilon ^{k}u_{k}\Big) dx}_{=I_1}
\underbrace{-\epsilon^{d+1}\int_{\Omega }A_{11}\nabla _{X_1}u_{d-1}\cdot \nabla _{X_1}\Big(u_{\epsilon
}-\sum_{k=0}^{d}\epsilon ^{k}u_{k}\Big) dx }_{=I_2} \\ \underbrace{-\epsilon^{d+1}\int_{\Omega }A_{12}\nabla_{X_2} u_d \cdot \nabla_{X_1} \Big(u_{\epsilon
}-\sum_{k=0}^{d}\epsilon ^{k}u_{k}\Big) dx}_{=I_3} - \underbrace{\epsilon^{d+1}\int_{\Omega }A_{21}\nabla_{X_1} u_d \cdot \nabla_{X_2} \Big(u_{\epsilon
}-\sum_{k=0}^{d}\epsilon ^{k}u_{k}\Big) dx}_{=I_4}.
\end{multline*} 
\begin{equation}\label{4-integral}
\end{equation}
By using Cauchy-Schwarz and Young's inequalities and assumption (\ref{hypA2}) we obtain 
\begin{equation}\label{I1}
I_1 \leq \frac{2}{\lambda}\epsilon ^{2d+2}\Vert A_{11} \Vert _{\infty}^2\left\Vert \nabla _{X_1}u_{d}\right\Vert _{L^{2}(\Omega)}^2  +\frac{\lambda}{4} \epsilon^2\Vert\nabla _{X_1}\Big( u_{\epsilon }-\sum_{k=0}^{d}\epsilon ^{k}u_{k}\Big) \Vert_{L^2(\Omega)}^2.
\end{equation}%
Similarly we derive
\begin{equation}\label{I2}
I_2 \leq \frac{2}{\lambda}\epsilon ^{2d}\Vert A_{11} \Vert _{\infty}^2\left\Vert \nabla _{X_1}u_{d-1}\right\Vert _{L^{2}(\Omega)}^2  +\frac{\lambda}{4} \epsilon^2\Vert\nabla _{X_1}\Big( u_{\epsilon }-\sum_{k=0}^{d}\epsilon ^{k}u_{k}\Big) \Vert_{L^2(\Omega)}^2,
\end{equation}%
and 
\begin{equation}\label{I4}
I_4 \leq \frac{2}{\lambda}\epsilon ^{2d+2}\Vert A_{21} \Vert _{\infty}^2\left\Vert \nabla _{X_1}u_{d}\right\Vert _{L^{2}(\Omega)}^2  +\frac{\lambda}{4} \Vert\nabla _{X_2}\Big( u_{\epsilon }-\sum_{k=0}^{d}\epsilon ^{k}u_{k}\Big) \Vert_{L^2(\Omega)}^2.
\end{equation}%

Now, we will estimate $I_3$:  Since $u_{\epsilon}-\sum_{k=0}^{d}\epsilon ^{k}u_{k} \in H_{0}^{1}(\Omega )$ and $\partial _{x_{i}}a_{ij}\in
L^{\infty }(\Omega )$, $\partial _{x_{j}}a_{ij}\in L^{\infty }(\Omega )$ for 
$i=1,...,q$ and $j=q+1,...,N,$ (thanks to assumption (\ref{hypAd1})) then we get by a direct density argument
that for $i=1,...,q$ and $j=q+1,...,N$: $$\ \partial _{x_{k}}(a_{ij}
(u_{\epsilon}-\sum_{k=0}^{d}\epsilon ^{k}u_{k}))\in L^{2}(\Omega ),$$ and%
\begin{equation*}
\partial _{x_{k}}(a_{ij}(u_{\epsilon}-\sum_{k=0}^{d}\epsilon ^{k}u_{k} ))=(u_{\epsilon}-\sum_{k=0}^{d}\epsilon ^{k}u_{k} )\partial _{x_{k}}a_{ij}+a_{ij}\partial _{x_{k}}(u_{\epsilon}-\sum_{k=0}^{d}\epsilon ^{k}u_{k} ), \text{ for }k=i,j.
\end{equation*}%
Whence 
\begin{multline*}
I_3 =-\epsilon^{d+1}
\sum_{i=1}^{q}\sum_{j=q+1}^{N}\int_{\Omega }a_{ij}\partial
_{x_{j}}u_{d}\partial _{x_{i}}(u_{\epsilon}-\sum_{k=0}^{d}\epsilon ^{k}u_{k})dx 
\\=-\epsilon^{d+1} \sum_{i=1}^{q}\sum_{j=q+1}^{N}\int_{\Omega }\partial
_{x_{i}}(a_{ij}(u_{\epsilon}-\sum_{k=0}^{d}\epsilon ^{k}u_{k}))\partial _{x_{j}}u_{d}dx 
+ \\ \epsilon^{d+1} \sum_{i=1}^{q}\sum_{j=q+1}^{N}\int_{\Omega} (u_{\epsilon}-\sum_{k=0}^{d}\epsilon ^{k}u_{k})\partial _{x_{i}}a_{ij}\partial _{x_{j}}u_{d}dx \\
=-\epsilon^{d+1} \sum_{i=1}^{q}\sum_{j=q+1}^{N}\int_{\Omega }\partial
_{x_{j}}(a_{ij}(u_{\epsilon}-\sum_{k=0}^{d}\epsilon ^{k}u_{k}))\partial _{x_{i}}u_{d}dx 
\\+\epsilon^{d+1} \sum_{i=1}^{q}\sum_{j=q+1}^{N}\int_{\Omega }(u_{\epsilon}-\sum_{k=0}^{d}\epsilon ^{k}u_{k})\partial _{x_{i}}a_{ij}\partial _{x_{j}}u_{d}dx,
\end{multline*}%
where we have used $$\int_{\Omega }\partial _{x_{i}}(a_{ij}(u_{\epsilon}-\sum_{k=0}^{d}\epsilon ^{k}u_{k}))\partial _{x_{j}}u_{d}dx=\int_{\Omega }\partial
_{x_{j}}(a_{ij}(u_{\epsilon}-\sum_{k=0}^{d}\epsilon ^{k}u_{k}))\partial _{x_{i}}u_{d}dx,$$ which
follows by a simple density argument (recall that $u_{d}\in H_{0}^{1}(\Omega
)$). Therefore%
\begin{multline*}
I_3=-\epsilon^{d+1}
\sum_{i=1}^{q}\sum_{j=q+1}^{N}\int_{\Omega }(u_{\epsilon}-\sum_{k=0}^{d}\epsilon ^{k}u_{k})\partial _{x_{j}}a_{ij}\partial _{x_{i}}u_{d}dx
 \\
-\epsilon^{d+1} \sum_{i=1}^{q}\sum_{j=q+1}^{N}\int_{\Omega }a_{ij}\partial
_{x_{j}}(u_{\epsilon}-\sum_{k=0}^{d}\epsilon ^{k}u_{k})\partial _{x_{i}}u_{d}dx \\
+\epsilon^{d+1} \sum_{i=1}^{q}\sum_{j=q+1}^{N}\int_{\Omega }(u_{\epsilon}-\sum_{k=0}^{d}\epsilon ^{k}u_{k})\partial _{x_{i}}a_{ij}\partial _{x_{j}}u_{d}dx.  \notag
\end{multline*}
By Young's inequality and Poincar\'{e}'s inequality (\ref{sob}), we obtain 
\begin{equation}\label{I3}
I_3 \leq \epsilon ^{2d+2} C_{A,\lambda}(\left\Vert \nabla _{X_1}u_{d}\right\Vert _{L^{2}(\Omega)}^2+\left\Vert \nabla _{X_2}u_{d}\right\Vert _{L^{2}(\Omega)}^2)  +\frac{\lambda}{4} \Vert\nabla _{X_2}\Big( u_{\epsilon }-\sum_{k=0}^{d}\epsilon ^{k}u_{k}\Big) \Vert_{L^2(\Omega)}^2.
\end{equation}%
Finally, we combine (\ref{4-integral}) with (\ref{I1}), (\ref{I2}), (\ref{I4}), (\ref{I3}) to get 
\begin{equation}
\frac{\lambda}{4} \epsilon ^{2}\left\Vert \nabla _{X_1}(u_{\epsilon }-\sum_{k=0}^{d}\epsilon^{k}u_{k})\right\Vert _{L^{2}(\Omega )}^{2}+ \frac{\lambda}{4}\left\Vert \nabla_{X_2}(u_{\epsilon }-\sum_{k=0}^{d}\epsilon^{k}u_{k})\right\Vert_{L^{2}(\Omega )}^{2} 
\leq C_{f,d,A,\lambda,\Omega} \epsilon^{2d}.
\end{equation}
Therefore, 
\begin{equation*}
\left\Vert \nabla _{X_2}(u_{\epsilon }-\sum_{k=0}^{d}\epsilon
^{k}u_{k})\right\Vert _{L^{2}(\Omega )}\leq  C_{d,f,A,\lambda,\Omega}\epsilon ^{d},
\end{equation*}%
and 
\begin{equation*}
\left\Vert \nabla _{X_1}(u_{\epsilon }-\sum_{k=0}^{d}\epsilon
^{k}u_{k})\right\Vert _{L^{2}(\Omega )} C_{d,f,A,\lambda,\Omega}\epsilon ^{d-1}.
\end{equation*}
And the proof of the theorem is finished.
\end{proof}

 \appendix
\section{Density Lemmas}

\begin{lemma}
\label{density1} $H_{0}^{m}(\omega _{1})\otimes H_{0}^{m}(\omega _{2})$ is
dense in $H_{0}^{m}(\Omega ).$
\end{lemma}

\begin{proof}
\textbf{1) }The embedding $\mathcal{D}(\Omega )\rightarrow H_{0}^{m}(\Omega
) $ is continuous thanks to the inequality, 
\begin{equation*}
\left( \sum_{\left\vert \gamma \right\vert =m}\int_{\Omega }\left\vert
D^{\gamma }u\right\vert ^{2}dx\right) ^{\frac{1}{2}}\leq mes(\Omega )\sup 
_{\substack{ x\in Supp(u)  \\ \left\vert \gamma \right\vert =m}}\left\vert
D^{\gamma }u\right\vert.
\end{equation*}

$\mathcal{D}(\omega _{1})\otimes \mathcal{D}(\omega _{2})$ is dense in $%
\mathcal{D}(\Omega )$, and $\mathcal{D}(\Omega )$ is dense in $%
H_{0}^{m}(\Omega )$ then $\mathcal{D}(\omega _{1})\otimes \mathcal{D}(\omega
_{2})$ is dense in $H_{0}^{m}(\Omega ),$ i.e. $H_{0}^{m}(\omega _{1})\otimes
H_{0}^{m}(\omega _{2})$ is dense in $H_{0}^{m}(\Omega ).$
\end{proof}

\begin{lemma}
\label{density2} $H_{0}^{m}(\Omega )$ is dense in $H_{0}^{m}(\Omega ;\omega
_{1})$
\end{lemma}

\begin{proof}
Let $u\in H_{0}^{m}(\Omega ;\omega _{1})$ fixed. For every $n\in 
\mathbb{N}
^{\ast }$ we denote $u_{n}$ the unique solution in $H_{0}^{m}(\Omega )$ of%
\begin{multline}
\sum_{\vert \alpha \vert=m}\int_{\Omega }D^{\alpha}_{X_1}
u_{n}\cdot D^{\alpha}_{X_1}\varphi dx+ \frac{1}{n^2}\sum_{\vert \gamma \vert+\vert \gamma'\vert=m, \text{} \gamma,\gamma' \neq 0}\int_{\Omega }D^{\gamma}_{X_1}D^{\gamma'}_{X_2}
u_{n}\cdot D^{\gamma}_{X_1}D^{\gamma'}_{X_2}\varphi dx\\+\frac{1}{n^2}\sum_{\vert \beta \vert=m}\int_{\Omega }D^{\beta}_{X_2}
u_{n}\cdot D^{\beta}_{X_2}\varphi dx=\sum_{\vert \alpha \vert=m}\int_{\Omega
}D _{X_1}^{\alpha}u\cdot D _{X_1}^{\alpha}\varphi dx,\text{ \ }\forall
\varphi \in H_{0}^{m}(\Omega ).  \label{elliptic-dordre 2m}
\end{multline}

The existence and uniqueness of $u_n$ follows from Lax-Milgram theorem. Testing with $%
u_{n}$ in (\ref{elliptic-dordre 2m}) we obtain for every multi-index $\alpha \in \mathbb{N}^q$ with $\vert \alpha\vert =m$:
\begin{equation*}
\left\Vert D _{X_1}^{\alpha}u_{n}\right\Vert _{L^{2}(\Omega )}\leq
\left\Vert D _{X_1}^{\alpha}v\right\Vert _{L^{2}(\Omega )},
\end{equation*}%
and for every multi-index $\gamma \in \mathbb{N}^q-\{0\}, \gamma'\in \mathbb{N}^{N-q}-\{0\}$ with $\vert \gamma \vert+\vert \gamma'\vert=m$:
\begin{equation*}
\frac{1}{n}\left\Vert D^{\gamma}_{X_1}D^{\gamma'}_{X_2}u_{n}\right\Vert _{L^{2}(\Omega )}\leq \left\Vert D _{X_1}^{\alpha}v\right\Vert _{L^{2}(\Omega )},
\end{equation*}
and for every multi-index $\beta \in \mathbb{N}^{N-q}$ with $\vert \beta \vert =m$:
\begin{equation*}
\frac{1}{n}\left\Vert D^{\beta}_{X_2}u_{n}\right\Vert _{L^{2}(\Omega )}\leq \left\Vert D _{X_1}^{\alpha}v\right\Vert _{L^{2}(\Omega )},
\end{equation*}
By using reflexivity of $L^{2}(\Omega )$ and continuity of the derivation operators on $%
\mathcal{D}^{\prime }(\Omega )$ we show that there exists $u_{\infty }\in
L^{2}(\Omega )$ such that $D _{X_1}^{\alpha}u_{\infty }\in L^{2}(\Omega )$ 
Fix multi-index $\alpha$, $\beta$, $\gamma$, $\gamma'$ given as above, and up to a subsequence we have   
\begin{equation}\label{mazur}
\frac{1}{n}D^{\beta}_{X_2}u_{n} \rightharpoonup 0\text{, }\frac{1}{n}D^{\gamma}_{X_1}D^{\gamma'}_{X_2}u_{n} \rightharpoonup 0%
\text{, and }D^{\alpha} _{X_1}u_{n}\rightharpoonup D^{\alpha} _{X_1}u_{\infty
}\text{ in }L^{2}(\Omega )\text{ weakly}.
\end{equation}
By using Mazur's lemma there exists a sequence $(U_n)$ of convex combinations of $(u_n)$ such that $$D^{\alpha} _{X_1}U_{n}\rightarrow D^{\alpha} _{X_1}u_{\infty
}\text{ in }L^{2}(\Omega ) \text{ strongly}.$$
By completeness of $H^m_0(\Omega,\omega_1)$ we get $u_{\infty} \in H^m_0(\Omega,\omega_1).$
Now, we will prove that $u=u_{\infty}$. Passing to the limit in (\ref{elliptic-dordre 2m}) as $n \rightarrow \infty$ by using (\ref{mazur}) we get 
 \begin{equation*}
\sum_{\vert \alpha \vert=m}\int_{\Omega }D^{\alpha}_{X_1}
u_{\infty}\cdot D^{\alpha}_{X_1}\varphi dx=\sum_{\vert \alpha \vert=m}\int_{\Omega
}D _{X_1}^{\alpha}u\cdot D _{X_1}^{\alpha}\varphi dx,\text{ \ }\forall
\varphi \in H_{0}^{m}(\Omega ).  
\end{equation*}
Take $\varphi =\varphi_1 \otimes \varphi_2$ with $\varphi_1 \in H^m_0(\omega_1), \text{and } \varphi_2 \in H^m_0(\omega_2)$, then we obtain, for a.e. $X_2 \in \omega_2:$
\begin{equation*}
\sum_{\vert \alpha \vert=m}\int_{\omega_1 }D^{\alpha}_{X_1}
u_{\infty}\cdot D^{\alpha}_{X_1}\varphi_1 dX_1=\sum_{\vert \alpha \vert=m}\int_{\omega_1
}D _{X_1}^{\alpha}u\cdot D _{X_1}^{\alpha}\varphi_1 dX_1,\text{ \ }\forall
\varphi_1 \in H_{0}^{m}(\omega_1 ).
\end{equation*}
Finally, by taking $\varphi_1=u(\cdot,X_2)-u_{\infty}(\cdot,X_2)$ in the above equation and integrating over $\omega_2$ we get 
$$ \sum_{\vert \alpha \vert=m}\Vert D^{\alpha}_{X_1}(u-u_{\infty}) \Vert _{L^2(\Omega)}^2=0, $$
and whence $u=u_{\infty}$, and $(U_n)$ converges to $u$ in $H^m_0(\Omega, \omega_1)$.
\end{proof}

\end{document}